\documentclass[11pt,english]{amsart}
\usepackage{csquotes}
\usepackage{pb-diagram}
\usepackage{graphicx}
\usepackage{amssymb}
\usepackage{abstract}
\usepackage{amsmath}
\usepackage{hyperref}
\usepackage[utf8]{inputenc}
\usepackage[letterpaper,margin=1in]{geometry}
\usepackage{latexsym,amsmath,amsfonts,amscd,amssymb}
\usepackage{graphicx}
\usepackage{pb-diagram}
\usepackage{graphicx}
\usepackage{csquotes}
\usepackage{pb-diagram}
\usepackage{graphicx}
\usepackage{amssymb}
\usepackage{abstract}
\usepackage{amsmath}
\usepackage{amsfonts}
\usepackage{color}
\usepackage{tikz}
\usepackage{eufrak}

\usepackage[all,cmtip]{xy}
\usepackage{cite}
\usepackage{amsmath}
\usepackage{hyperref}
\newcommand\addtag{\refstepcounter{equation}\tag{\theequation}}

\theoremstyle{plain}  
\newtheorem{theorem}{Theorem}[section]
\newtheorem*{theorem*}{Theorem}

\newtheorem{lemma}[theorem]{Lemma}

\theoremstyle{definition}
\newtheorem{definition}[theorem]{Definition}

\theoremstyle{plain}

\newtheorem{remark}[theorem]{Remark}

\newtheorem*{claim*}{Claim}

\numberwithin{equation}{section}

\renewcommand{\leq}{\leqslant}

\renewcommand{\geq}{\geqslant}

\newcommand{\R}{\mathbb{R}}

\newcommand{\C}{\mathbb{C}}

\newcommand{\U}{\mathrm{U}}

\DeclareMathOperator{\rk}{rk}

\DeclareMathOperator{\coker}{coker}
\DeclareMathOperator{\Hom}{Hom}

\DeclareMathOperator{\Id}{Id}

\usepackage{fancyhdr}
\usepackage{mathdots}
\DeclareMathAlphabet{\mathpzc}{OT1}{pzc}{m}{it}

\title{Quiver bundles and wall crossing for chains}

\author[P. B. Gothen]{P. B. Gothen}
\address{Centro de
  Matem\'atica da Universidade do Porto \\
Faculdade de Ci\^encias da Universidade do Porto \\
Rua do Campo Alegre, s/n \\ 4169-007 Porto \\ Portugal }
\email{pbgothen@fc.up.pt}

\author[A. Nozad]{A. Nozad}
\address{Centro de Matem\'atica, Aplica\c{c}\~oes Fundamentais e Investiga\c{c}\~ao Operacional\\
  Faculdade de Ci\^encias da Universidade de Lisboa\\
  Edf. C6, Campo Grande \\1749-016 Lisboa\\ Portugal}

\email{anozad@fc.ul.pt}

\thanks{Partially supported by CMUP (UID/MAT/00144/2013) (first
  author), CMAF-CIO (UID/MAT/04561/2013) and grant SFRH/BD/51166/2010
  (second author), and the project PTDC/MAT-GEO/2823/2014 (both
  authors) funded by FCT (Portugal) with national funds.
  The authors acknow\-ledge support from U.S. National Science
  Foundation grants DMS 1107452, 1107263, 1107367 "RNMS: Geometric
  structures And Representation varieties" (the GEAR Network)}


\begin{document}

\maketitle

\begin{abstract}
  Holomorphic chains on a Riemann surface arise naturally as fixed
  points of the natural $\C^*$-action on the moduli space of Higgs
  bundles. In this paper we associate a new quiver bundle to the
  $\Hom$-complex of two chains, and prove that stability of the chains
  implies stability of this new quiver bundle. Our approach uses the
  Hitchin-Kobayashi correspondence for quiver bundles. Moreover, we use
  our result to give a new proof of a key lemma on chains (due to
  \'{A}lvarez-C\'{o}nsul--Garc{\'{\i}}a-Prada--Schmitt), which has
  been important in the study of Higgs bundle moduli; this proof relies
  on stability and thus avoids the direct use of the chain vortex
  equations.
\end{abstract}

\section{Introduction}
A \emph{holomorphic $(m+1)$-chain} on a compact Riemann surface $X$ of
genus $g\geq 2$ is a diagram
\[
\xymatrix{
C:\mbox{ }E_m\ar[r]^{\phi_m}&E_{m-1}\ar[r]^{\phi_{m-1}}&\cdots\ar[r]^{\phi_2}&E_1\ar[r]^{\phi_1}&E_0,\\
}
\]
where each $E_i$ is a holomorphic vector bundle and
$\phi_i: E_i\longrightarrow E_{i-1}$ is a holomorphic map. Moduli
spaces for holomorphic chains have been constructed by Schmitt
\cite{Schmitt:2005} using GIT and, as is usual for decorated bundles,
depend on a stability parameter
$\boldsymbol{\alpha}=(\alpha_0,\dots,\alpha_m)$, where
$\alpha_i\in\R$.

One important application of holomorphic chains stems from the fact
that, for a specific value of the stability parameter, their moduli
can be identified with fixed loci for the natural $\C^*$-action on the
moduli space of Higgs bundles. Thus, knowledge of moduli spaces of
chains can be used to study the moduli space of Higgs bundles. The
basic idea (in the case of rank 2 Higgs bundles) goes back to the
seminal paper of Hitchin \cite{hitchin:1987a}. 

For higher rank Higgs
bundles, knowledge of the moduli of chains becomes in itself difficult
to come by, and a successful strategy for this has been to study the
variation of the moduli of chains under changes in the parameter,
using wall crossing arguments. This approach goes back to the work of
Thaddeus \cite{thaddeus:1994} (used for rank 3 Higgs bundles in
\cite{gothen:1994}). 
Recent important examples of the study of wall crossing of chains and
applications to moduli of Higgs bundles include the work of
García-Prada--Heinloth--Schmitt \cite{garcia-heinloth-schmitt:2014},
García-Prada--Heinloth \cite{garcia-heinloth:2013}, and Heinloth (see
also Bradlow--García-Prada--Gothen--Heinloth \cite{bradlow-etal:2017}
for an application to $\U(p,q)$-Higgs bundles). We should mention
here that
recently alternative approaches to the study of the cohomology of
Higgs bundle moduli 
have been highly succesful: see Schiffman \cite{schiffmann:2016},
Mozgovoy--Schiffman \cite{mozgovoy-schiffman:2017} and Mellit
\cite{mellit:2017}; also, Maulik--Pixton have announced a proof of a
conjecture of Chuang--Diaconescu--Pan \cite{CDP} which leads to a calculation
of the motivic class of the moduli space of twisted Higgs bundles.

All the aforementioned
results on chains rely on a key result of
\'{A}lvarez-C\'{o}nsul--Garc{\'{\i}}a-Prada--Schmitt
\cite[Proposition~4.14]{Consul-Prada-Schmitt:2006} which, in
particular, is used in estimating codimensions of flip loci under wall
crossing. The proof of this result is analytic in nature and relies on
the solutions to the chain vortex equations, whose existence is
guaranteed by the Hitchin--Kobayashi correspondence for holomorphic
chains (see \'{A}lvarez-C\'{o}nsul--Garc{\'{\i}}a-Prada
\cite{Prada:2001,Prada:2003}).

In this paper, given a pair of chains, we associate to them a new
quiver bundle which extends and refines the $\Hom$-complex of the
chains; we call it the \emph{extended $\Hom$-quiver}. Moreover, we
show that polystability of the chains implies polystability of this
extended $\Hom$-quiver (see Theorem~\ref{relating polystability}). We
then use our result to give a new and simpler proof of the key result
\cite[Proposition~4.14]{Consul-Prada-Schmitt:2006} mentioned above
(see Theorem~\ref{thm:AGS}). The main merit of our argument is that it
is algebraic, in the sense that it only uses stability of the extended
$\Hom$-quiver and avoids direct use of the chain vortex
equations. Thus, though our proof of Theorem~\ref{relating
  polystability} does ultimately rely on the Hitchin--Kobayashi
correspondence (through Lemma~\ref{lem:hom-complex-vortex}), the roles
of the correspondence and of stability are clarified. Our result can
be viewed as a generalization of a result of
\cite{bradlow-garcia-prada-gothen:2004} for length two chains (also
known as triples), though in this case the extended $\Hom$-quiver is
itself a chain.

\subsection*{Acknowledgments}
We thank Steve Bradlow for useful discussions and we thank the referee for
insightful comments which helped improve the exposition.

\section{Definitions and basic results}
\label{sec:basic-results}
In this section we recall definitions and relevant facts on quiver bundles, from
\cite{Gothen:2005} and \cite{Prada:2003}. 

\subsection{Quiver bundles} A \emph{quiver} $Q$ is a directed graph
specified by a set of vertices $Q_0$, a set of arrows $Q_1$ and head
and tail maps $h,t : Q_1\to Q_0$. We shall assume that $Q$ is finite.  

\begin{definition}A holomorphic \emph{quiver bundle}, or simply a
  \emph{$Q$-bundle}, is a pair $\mathcal{E} = (V, \varphi)$, where $V$
  is a collection of holomorphic vector bundles $V_i$ on $X$, for each
  $i \in Q_0$, and $\varphi$ is a collection of morphisms
  $\varphi_a: V_{ta}\to V_{ha}$, for each $a\in Q_1$.
\end{definition}

The notions of $Q$-subbundles and quotient $Q$-bundles, as well as
simple $Q$-bundles are defined in the obvious way. The subobjects
$(0,0)$ and $\mathcal{E}$ itself are called the \emph{trivial
  subobjects}. The \emph{type} of a $Q$-bundle
$\mathcal{E}= (V, \varphi)$ is given by
$$t(\mathcal{E})=(\rk(V_i); \deg(V_i))_{ i\in Q_0},$$ where $\rk(V_i)$
and $\deg(V_i)$ are the rank and degree of $V_i$, respectively. We
sometimes write $\rk(\mathcal{E})= \rk(\bigoplus V_i)$ and call it the
\emph{rank of $\mathcal{E}$}. Note that the type is independent of
$\varphi$.

\subsection{Stability}
Fix a tuple $\boldsymbol{\alpha}=(\alpha_i)\in\mathbb{R}^{|Q_0|}$ of
real numbers. For a non-zero $Q$-bundle $\mathcal{E}=(V,\varphi)$, the
associated \emph{$\boldsymbol{\alpha}$-slope} is defined as
$$\mu_{\boldsymbol{\alpha}}(\mathcal{E})=\frac{\underset{i\in Q_0}{\sum}\big(\alpha_i \rk(V_i)+\deg(V_i)\big)}{\underset{i\in Q_0}{\sum}\rk(V_i)}.$$

\begin{definition}
  A $Q$-bundle $\mathcal{E}=(V,\varphi)$ is said to be
  \emph{$\boldsymbol{\alpha}$-(semi)stable} if, for all non-trivial
  subobjects $\mathcal{F}$ of ${\mathcal{E}}$,
  $\mu_{\boldsymbol{\alpha}}(\mathcal{F})
    <(\leq)\mu_{\boldsymbol{\alpha}}(\mathcal{E})$. An
  \emph{$\boldsymbol{\alpha}$-polystable} $Q$-bundle is a finite
  direct sum of $\boldsymbol{\alpha}$-stable $Q$-bundles, all of them
  with the same $\boldsymbol{\alpha}$-slope.

  A $Q$-bundle $\mathcal{E}$ is \emph{strictly
    $\boldsymbol{\alpha}$-semistable} if there is a
  non-trivial subobject $\mathcal{F}\subset \mathcal{E}$ such that
  $\mu_{\boldsymbol{\alpha}}(\mathcal{F})=\mu_{\boldsymbol{\alpha}}(\mathcal{E})$.
\end{definition}

\begin{remark}
  In fact, the most general stability condition for quiver bundles
  involves additional parameters, see
  \cite{Consul-Prada-Schmitt:2006}. Since 
  $\boldsymbol{\alpha}$ is the parameter which has been used in the
  literature for the study of moduli of chains via wall crossing, 
  we confine ourselves to considering this parameter.
\end{remark}

\subsection{The gauge theory equations}

Let $\mathcal{E}=(V,\varphi)$ be a $Q$-bundle on $X$. A Hermitian
metric on $\mathcal{E}$ is a collection $H$ of Hermitian metrics $H_i$
on $V_i$, for each $i\in Q_0$. To define the gauge
equations on $\mathcal{E}$, we note that $\varphi_a:V_{ta}\to V_{ha}$
has a smooth adjoint morphism $\varphi_a^\ast:V_{ha}\to V_{ta}$ with
respect to the Hermitian metrics $H_{ta}$ on $V_{ta}$ and $H_{ha}$ on
$V_{ha}$, for each $a\in Q_1$, so it makes sense to consider the
compositions $\varphi_a\circ\varphi_a^\ast$ and
$\varphi_a^\ast\circ\varphi_a$.

Let $\boldsymbol{\alpha}$ be the stability parameter. Define
$\boldsymbol{\tau}$ to be the vector of real numbers $\tau_i$ given by
\begin{equation}
\tau_i=\mu_{\boldsymbol{\alpha}}(\mathcal{E})-\alpha_i, \mbox{ }i\in Q_0.\label{relation}
\end{equation}
Since the stability condition does not change under a global translation $\boldsymbol{\alpha}$ can be recovered from $\boldsymbol{\tau}$ as follows
\begin{align*}
\alpha_i=\tau_{0}-\tau_i, \mbox{ } i\in Q_0.
\end{align*}
\begin{definition}
A Hermitian metric $H$ satisfies the \emph{quiver $\boldsymbol{\tau}$-vortex equations} if
\begin{equation}
\sqrt{-1}\Lambda F(V_i)+\underset{i=ha}{\sum}\varphi_a\varphi_a^{\ast}-\underset{i=ta}{\sum}\varphi_a^{\ast}\varphi_a=\tau_i \mathrm{Id}_{V_i}\label{vortex equations1}
\end{equation}
for each $i\in Q_0$, where $F(V_i)$ is the
curvature of the Chern connection associated to the metric $H_i$ on
the holomorphic vector bundle $V_i$, and $ \Lambda: \Omega^{i,j}(M)\to \Omega^{i-1,j-1}(M)$ is the contraction operator with respect to a fixed K\"{a}hler form $\omega$ on $X$. 
\end{definition}

The following is the \emph{Hitchin--Kobayashi correspondence} between
the twisted quiver vortex equations and the stability condition for
holomorphic twisted quiver bundles, given by Álvarez-Cónsul and
García-Prada \cite[Theorem~3.1]{Prada:2003}:
\begin{theorem}\label{Hitchin-Kobayashi}
A holomorphic $Q$-bundle $\mathcal{E}$  is $\boldsymbol{\alpha}$-polystable if and only if it admits a Hermitian metric $H$ satisfying the quiver $\boldsymbol{\tau}$-vortex equations $(\ref{vortex equations1})$, where $\boldsymbol{\alpha}$ and $\boldsymbol{\tau}$ are related by $(\ref{relation})$. 
\end{theorem}
Note that the definitions and facts can be specialized for holomorphic chains. 

\subsection{The $\Hom$-complex for chains}
Fix two holomorphic chains $C''$ and $C'$, given by
\[
\xymatrix{
C':\mbox{ }E'_m\ar[r]^{\phi'_m}&E'_{m-1}\ar[r]^{\phi'_{m-1}}&\cdots\ar[r]^{\phi'_2}&E'_1\ar[r]^{\phi'_1}&E'_0\\
}
\]
\[
\xymatrix{
C'':\mbox{}E''_m\ar[r]^{\phi''_m}&E''_{m-1}\ar[r]^{\phi''_{m-1}}&\cdots\ar[r]^{\phi''_2}&E''_1\ar[r]^{\phi''_1}&E''_0\\
}
\]
Consider the following two terms complex of sheaves 
\begin{equation}\label{deformation complex of chains}
\mathcal{H}^{\bullet}(C'',C'): \mathcal{H}^0\overset{d}{\longrightarrow} \mathcal{H}^{1}
\end{equation}
with terms
\begin{eqnarray*}
& &\mathcal{H}^0=\bigoplus_{i-j=0} \Hom(E''_i,E'_j),\mbox{ }\mathcal{H}^{1}=\bigoplus_{i-j=1}\Hom( E''_i, E'_j),
\end{eqnarray*}
and the map $d$ is defined
by $$d(g_0,\ldots,g_m)=(g_{i-1}\circ\phi''_i-\phi'_i\circ g_i),\mbox{
  for }g_i\in \Hom(E''_i,E'_i).$$  The complex $\mathcal{H}^{\bullet}(C'',C')$
is called the \emph{$\Hom$-complex}. It governs the homological algebra of
chains; in particular $\mathcal{H}^{\bullet}(C,C)$ is the deformation
complex of a chain $C$.

\section{The extended $\Hom$-quiver}

Here we introduce a $Q$-bundle, associated to 
two chains, and show that solutions to the vortex equations on the
holomorphic chains produce a solution on the corresponding quiver
bundle. The basic idea is the following: to the chains $C''$ and $C'$ we
associate the vector bundles $E''$ and $E'$, obtained as the direct sum
of the individual bundles in the chains. The quiver bundle structure
on the chains then induces a natural quiver bundle structure on the
bundle $\Hom(E'',E')$. Thus our construction can be seen as a
kind of extension of structure group and it becomes natural to expect that a solution to the
vortex equations on the chains should give a solution on the induced
quiver bundle. Indeed, this is exactly the
content of our Lemma~\ref{lem:hom-complex-vortex} below. This in turn
implies the main result of this section,
Theorem~\ref{relating polystability}, which says that
$\tilde{\mathcal{H}}(C'',C')$ is (poly)stable for suitable values of
the parameter.

We note that, since there are algebraic proofs of results saying that
stability is preserved under extension of structure group (in the
setting of principal bundles by Ramanan--Ramanathan \cite{RR} and for Hitchin
pairs by Balaji--Parameswaran \cite{BP2012}) this might indicate the possibility of
an algebraic proof of our result as well, though we do not pursue
this possibility here.

We also point out that one might attempt to generalize our
construction to more general quiver bundles than chains; see
\cite[Section~4.2]{GN} for the case of $\U(p,q)$-Higgs bundles.

\begin{definition}
  Let $C'$ and $C''$ be chains of length $m$. The \emph{extended
    $\Hom$-quiver} $\tilde{\mathcal{H}}(C'',C')$ is a quiver bundle
  defined as follows:
  \begin{itemize}
  \item For each $(i,j)$ with $0\leq i,j\leq m$, there is a vertex to
    which we associate the bundle $\Hom(E_i'',E_j')$, of \emph{weight
      $k=i-j$}. 
  \item For each $\Hom(E_i'',E_j')$, of weight
      $k=i-j$, there are maps
      \begin{align*}
        \delta_{ij}^-\colon\Hom(E_i'',E_j') &\to \Hom(E_i'',E_{j-1}'),\\
        f &\mapsto -\phi_j'\circ f,
      \end{align*}
      and
      \begin{align*}
        \delta_{ij}^+\colon\Hom(E_i'',E_j') &\to \Hom(E_{i+1}'',E_{j}'),\\
        f &\mapsto f\circ\phi_{i+1}''.
      \end{align*}
  \end{itemize}
\end{definition}
In other words, $\tilde{\mathcal{H}}(C'',C')$ is defined by
associating to $(E''=\bigoplus E_i'',\phi''=\sum_i\phi_i'')$ and
$(E'=\bigoplus E_i',\phi'=\sum_i\phi_i')$ the bundle $\Hom(E'',E')$
and the map
\begin{align*}
  \Hom(E'',E') &\to \Hom(E'',E'),\\
  f &\mapsto f\circ \phi''-\phi'\circ f,
\end{align*}
and then taking the quiver bundle induced from the splitting
$\Hom(E'',E') = \bigoplus_{i,j} \Hom(E_i'',E_j')$.
We can picture this construction as follows:

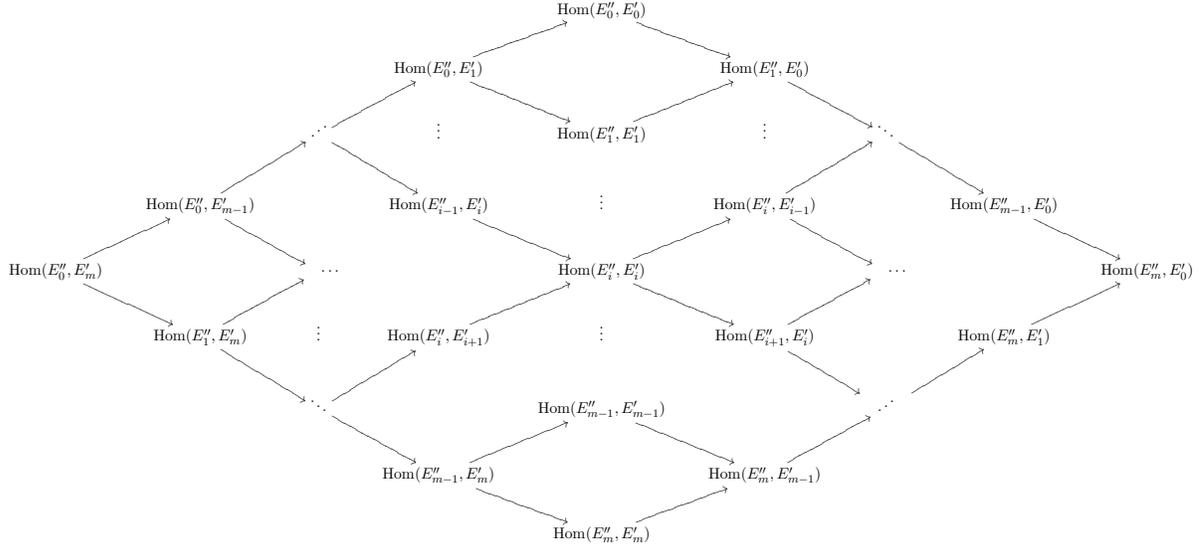
\begin{equation}\addtag\label{quiver}
\resizebox{\linewidth}{!}{
 \xymatrix{
   & & & && \Hom(E''_0,E_0')\ar[dr]& & & &\\
   & & & &\Hom(E_{0}'', E_1')\ar[ur]\ar[dr]& &\Hom(E_1'',E_0')\ar[dr]& & &\\
      & & & \iddots\ar[ur]\ar[dr]& \vdots& \Hom(E''_1,E_1')\ar[ur]& \vdots& \ddots\ar[dr]&& \\
        & &\Hom(E_0'',E'_{m-1})\ar[ur]\ar[dr]& &\Hom(E_{i-1}'', E_i')\ar[dr]&\vdots &\Hom(E_{i}'',E_{i-1}')\ar[ur] \ar[dr]&&\Hom(E_{m-1}'',E'_0)\ar[dr] \\
          &\Hom(E_0'',E'_{m})\ar[ur]\ar[dr]& & \mbox{ }\mbox{ }\mbox{ }\mbox{  }\cdots& &\Hom(E_i'', E_i')\ar[ur]\ar[dr]& & \mbox{ }\mbox{ }\mbox{ }\mbox{  }\cdots & &\Hom(E_m'', E'_0)\\
            & & \Hom(E_1'',E'_{m})\ar[ur]\ar[dr]&\vdots&\Hom(E_{i}'',E_{i+1}')\ar[ur]&\vdots & \Hom(E_{i+1}'',E_i') \ar[ur]\ar[dr]&&  \Hom(E_m'',E'_{1})\ar[ur]& \\
                     & & &\ddots\ar[dr]\ar[ur] &&\Hom(E_{m-1}'',E'_{m-1})\ar[dr]&& \mbox{ }\mbox{ } \iddots \mbox{ }\mbox{ }\ar[ur] && \\
                    & & & &\Hom(E_{m-1}'', E'_m)\ar[ur]\ar[dr]&& \Hom( E_{m}'',E'_{m-1})\ar[ur]&&&\\
                       & & &&&\Hom(E_m'',E'_m)\ar[ur]& & &&}}
\end{equation}

Note that if we take the direct sums of the middle two columns
$${\textstyle\underset{i-j=0}{\bigoplus}\Hom(E_i'',E'_j)
  \xrightarrow{\delta^++\delta^-}
  \underset{i-j=1}{\bigoplus}\Hom(E_i'', E'_j)}
$$ 
we obtain 
the $\Hom$-complex of the chains $C''$ and $C'$, defined in
$(\ref{deformation complex of chains})$.

\begin{lemma}
  \label{lem:hom-complex-vortex}
  Let $C'$ and $C''$ be holomorphic chains and suppose we have
  solutions to the $(\tau'_0,\ldots,\tau'_m)$-vortex equations on $C'$
  and the $(\tau''_0,\ldots,\tau_m'')$-vortex equations on $C''$. Then
  the induced Hermitian metric on the extended $\Hom$-quiver
  $\tilde{\mathcal{H}}(C'',C')$ pictured in
  $(\ref{quiver})$ satisfies the quiver
  $\widetilde{\boldsymbol{\tau}}$-vortex equations, for
  $\widetilde{\boldsymbol{\tau}}
  =(\tilde{\tau}_{ij})
  =(\tau'_j-\tau^{''}_{i})$.
\end{lemma}

\begin{proof}
  To show that the induced Hermitian metric satisfies the equation at
  $\Hom(E_{i}'',E'_j)$ of weight $k$, for $-m\leq k\leq m$,
  first recall that we have the following identity of curvature
  operators:
  $$F\big(\Hom(E_{i}'', E'_j)\big)(f)=F(E'_j)\circ f -f\circ
  F(E''_{i}).$$
Also, the vortex equations for $C'$ and $C''$ are
$$\sqrt{-1}\Lambda F(E_i')+\phi_{i+1}'\phi_{i+1}^{'\ast}-\phi_i^{'\ast}\phi'_i=\tau_i' \Id_{E'_i},\mbox{ }i=0,\ldots,m$$
$$\sqrt{-1}\Lambda F(E^{''}_i)+\phi^{''}_{i+1}\phi^{''\ast}_{i+1}-\phi_i^{\ast''}\phi_i^{''}=\tau_i'' \Id_{E^{''}_i}\mbox{,  }i=0,\ldots,m.$$
Now, considering the quiver $\tilde{\mathcal{H}}(C'',C')$ at
$\Hom(E_{i}'',E'_j)$ we have
\[
 \xymatrix@=0.5em@1@R=1.5em{
 &\Hom(E_{i-1}'',E'_j)\ar[dr]^{\delta_c}& &\Hom(E_{i}'',E'_{j-1})\\
 &&\Hom(E_{i}'',E'_{j})\ar[ur]^{\delta_b}\ar[dr]_{\delta_a}&\\
  &\Hom(E_{i}'',E'_{j+1})\ar[ur]_{\delta_d}& &\Hom(E_{i+1}'',E'_j)}
  \]
 where for ease of notation we have written
 \begin{eqnarray*}
 \delta_a(f) &=& \delta_{ij}^+(f)= f\circ\phi_{i+1}''\\
 \delta_b(f) &=& -\delta_{ij}^-(f)= \phi'_j\circ f\\
 \delta_c(g) &=& \delta_{i-1,j}^+(g)=g\circ\phi_{i}''\\
 \delta_d(h) &=& -\delta_{i,j+1}^-(h)= \phi_{j+1}'\circ h
 \end{eqnarray*}
A straightforward calculation gives the following 
  \begin{eqnarray*}
   \delta_a^\ast(g)&=& g\circ\phi_{i+1}^{''\ast}\\
  \delta_b^\ast(h)&=& \phi^{'\ast}_j\circ h\\
  \delta_c^\ast(f)&=&f\circ\phi_{i}^{''\ast}\\
  \delta_d^\ast(f)&=& \phi_{j+1}^{'\ast}\circ f
    \end{eqnarray*}
therefore
\begin{eqnarray*}
\left(\delta_c\delta_c^\ast+\delta_d\delta_d^\ast-\delta_a^\ast\delta_a-\delta_b^\ast\delta_b\right)(f)
&=&
\delta_c(f\circ\phi_{i}^{''\ast})+\delta_d(\phi_{j+1}^{'\ast}\circ f)-\delta_a^\ast(f\circ\phi_{i+1}'')-\delta_b^\ast(\phi'_j\circ f)\\
&=&
f\circ\phi_{i}^{''\ast}\circ\phi_{i}''+\phi_{j+1}'\circ\phi_{j+1}^{'\ast}\circ f-f\circ\phi_{i+1}''\phi_{i+1}^{''\ast}-\phi^{'\ast}_j\phi'_j\circ f
\end{eqnarray*}
Hence, using the vortex equations for $C'$ and $C''$ and the above
identity of curvature operators, we have for $f\in\Hom(E_{i+k}'', E'_i)$:
\begin{eqnarray*}
\lefteqn{(\sqrt{-1}\Lambda F(\Hom(E_{i}'', E'_j))+\delta_c\delta_c^\ast+\delta_d\delta_d^\ast-\delta_a^\ast\delta_a-\delta_b^\ast\delta_b)(f)}\\
&=&
\Big(\big(\sqrt{-1}\Lambda F(E'_j)+\phi^{''}_{j+1}\phi_{j+1}^{''\ast}-\phi_j^{''\ast}\phi^{''}_j\big)\circ f-f\circ\big( \sqrt{-1}\Lambda F(E''_{i})-\phi_{i}'\phi_{i}^{'\ast}+\phi_{i+1}^{'\ast}\phi_{i+1}'\big)\Big)\\
&=&
(\tau_j'-\tau''_{i})f.
\end{eqnarray*}
This finishes the proof.
\end{proof}

\begin{theorem}\label{relating polystability}
 Let $C'$ and $C''$ be $\boldsymbol{\alpha}'=(\alpha'_1,\ldots,\alpha'_m)$ and $\boldsymbol{\alpha}''=(\alpha''_1,\ldots,\alpha''_m)$-polystable holomorphic chains, respectively. Then the extended $\Hom$-quiver
  $\tilde{\mathcal{H}}(C'',C')$, as in
  $(\ref{quiver})$, is $\boldsymbol{\widetilde{\alpha}}=(\widetilde{\alpha}_{ij})$-polystable for $
\widetilde{\alpha}_{ij}=\alpha''_m+\alpha'_j-\alpha''_{i}$.
\end{theorem}

\begin{proof}
Since the holomorphic chains $C'$ and $C''$ are
$\boldsymbol{\alpha}'$- and $\boldsymbol{\alpha}''$-polystable, it
follows from Proposition~\ref{Hitchin-Kobayashi} that both
the $(\tau'_0,\ldots,\tau'_m)$- and the $(\tau''_0,\ldots,\tau''_m)$-vortex
equations have a solution. Then, by Lemma~\ref{lem:hom-complex-vortex}
the extended $\Hom$-quiver $\tilde{\mathcal{H}}(C'',C')$ satisfies the
quiver $(\tau'_j-\tau^{''}_{i})$-vortex equations and therefore the Hitchin--Kobayashi correspondence implies that $\tilde{\mathcal{H}}(C'',C')$ is $\widetilde{\boldsymbol{\alpha}}$-polystable for
\begin{eqnarray*}
\widetilde{\alpha}_{ij}&=&\tau'_0-\tau''_m-(\tau'_{j}-\tau^{''}_i)=\tau'_0-\tau'_{j}+\tau''_0-\tau''_m+\tau''_i-\tau''_0=\alpha_m''+\alpha_{j}'-\alpha_i''.
\end{eqnarray*}
\end{proof}

\section{Application to wall crossing for chains}

As an application of Theorem~\ref{relating polystability} we give a
simplified and more conceptual proof of a result of
\'{A}lvarez-C\'{o}nsul, Garc{\'{\i}}a-Prada and Schmitt in
\cite{Consul-Prada-Schmitt:2006}, showing how it follows from
stability of the quiver bundle (\ref{quiver1}).  This result is a key
ingredient in wall crossing arguments for holomorphic chains, which
have had a number of important applications lately as explained in the
introduction. First we state a particular case of our main theorem
which will be used in the proof.

If we take
$\boldsymbol{\alpha}=\boldsymbol{\alpha}'=\boldsymbol{\alpha}''$ in
Theorem~\ref{relating polystability}, then the stability parameter at
every vertex in the middle column of \eqref{quiver} is
$\tilde{\alpha}_{ii} = \alpha_m+\alpha_i-\alpha_i=\alpha_m$. Hence we
can collapse the central column in the quiver into a single vertex, to
which we associate the direct sum of the corresponding bundles and
obtain the following quiver
bundle: 
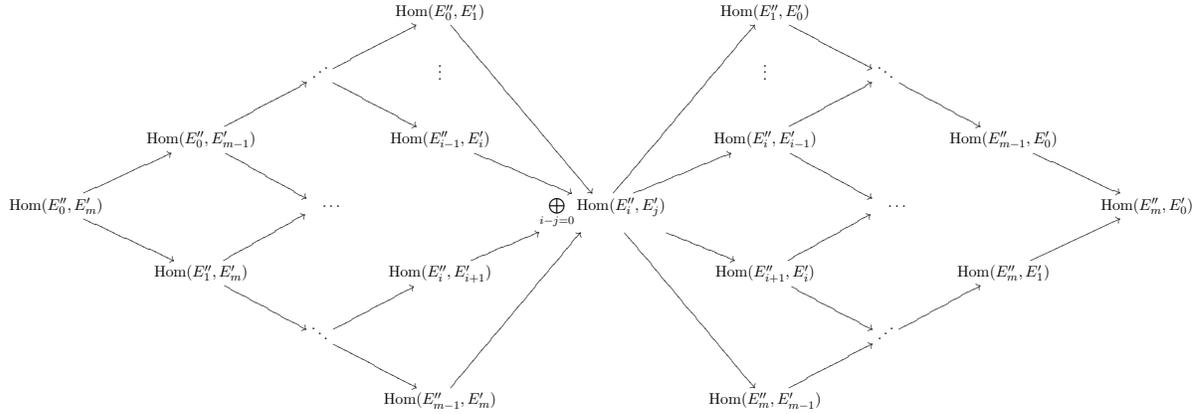
\begin{equation}\label{quiver1}
\resizebox{\linewidth}{!}{
\xymatrix{
& & &                      &\Hom(E_{0}'', E_1')\ar[dddr] & &\Hom(E_1'',E_0')\ar[dr]&               & &\\
& & & \iddots\ar[ur]\ar[dr]& \vdots                     & & \vdots                & \ddots\ar[dr] & &\\
& & \Hom(E_0'',E'_{m-1})\ar[ur]\ar[dr]& &\Hom(E_{i-1}'', E_i')\ar[dr]& &\Hom(E_{i}'',E_{i-1}')\ar[ur] \ar[dr]&&\Hom(E_{m-1}'',E'_0)\ar[dr] \\
& \Hom(E_0'',E'_{m})\ar[ur]\ar[dr]& & \mbox{ }\mbox{ }\mbox{ }\mbox{  }\cdots& & \underset{i-j=0}{\bigoplus}\Hom(E_i'', E_j')\ar[ur]\ar[dr]\ar[uuur]\ar[dddr]& & \mbox{ }\mbox{ }\mbox{ }\mbox{  }\cdots & &\Hom(E_m'', E'_0)\\
& &   \Hom(E_1'',E'_{m})\ar[ur]\ar[dr]&& \Hom(E_{i}'',E_{i+1}')\ar[ur]& & \Hom(E_{i+1}'',E_i')\ar[ur]\ar[dr]&& \Hom(E_m'',E'_{1})\ar[ur]& \\
& & &\ddots\ar[ur] \ar[dr]&& && \iddots\ar[ur] && \\
& & & &\Hom(E_{m-1}'', E'_m)\ar[uuur]&&  \Hom( E_{m}'',E'_{m-1})\ar[ur]&& }}
\end{equation}
The next theorem says that this will be an $\bar{\boldsymbol\alpha}$-polystable quiver bundle for
the corresponding collapsed stability
parameter $\bar{\boldsymbol\alpha}$.

\begin{theorem}\label{cor:relating-polystability}
Let $C'$ and $C''$ be
$\boldsymbol{\alpha}=(\alpha_1,\ldots,\alpha_m)$-polystable
holomorphic chains. Then the quiver bundle pictured in
\eqref{quiver1} is $\bar{\boldsymbol\alpha}$-semistable, where the
stability parameter $\bar{\boldsymbol\alpha}$ is defined by 
$\bar{\alpha}_{ij}=\alpha_m+\alpha_j-\alpha_{i}$ (at the central
vertex we mean by this that the parameter is $\alpha_m$).
\end{theorem}

\begin{proof}
Any quiver subbundle $F$ of \eqref{quiver}
induces a quiver subbundle of \eqref{quiver1} by collapsing the middle
column and, by our assumption on the stability parameters, the
$\bar{\boldsymbol\alpha}$-slope of the collapsed quiver bundle $F$ equals
the $\boldsymbol\alpha$-slope of the original quiver bundle. 
Thus quiver subbundles $F$ of \eqref{quiver1} 
obtained from quiver subbundles of \eqref{quiver} by collapsing the
middle column satisfy the $\bar{\boldsymbol\alpha}$-semistability
condition.
This in fact suffices to prove the result by using a
standard argument (see, e.g., \cite[Proposition~3.11]{Prada:2001} or
\cite[Lemma~2.2]{gothen:1994}): the idea is to use that any quiver
subbundle of \eqref{quiver1} can be obtained by successive extensions of 
quiver subbundles of \eqref{quiver}.
\end{proof}

\begin{remark}
  An alternative proof of Theorem~\ref{cor:relating-polystability}
  (allowing to conclude polystability rather than semistability) can
  be given using the Hitchin--Kobayashi correspondence. Simply note
  that by Lemma~\ref{lem:hom-complex-vortex} a solution to the
  vortex equations on $C'$ and $C''$ gives a solution on the quiver
  bundle \eqref{quiver}.  Under the assumption on the parameters this,
  in turn, gives a solution on the collapsed quiver bundle \eqref{quiver1}.
\end{remark}

\begin{theorem}[{\'{A}lvarez-C\'{o}nsul--Garc{\'{\i}}a-Prada--Schmitt \cite[Proposition~4.4]{Consul-Prada-Schmitt:2006}}]
\label{thm:AGS}
Let $C'$ and $C^{''}$ be $\boldsymbol{\alpha}$-polystable holomorphic
chains and let $\alpha_{i}-\alpha_{i-1}\geq 2g-2$ for all $i=1,
\cdots, m$. Then the following inequalities hold
\begin{align}\label{kernel(d)}
\mu(\ker(d))&\leq\mu_{\boldsymbol\alpha}(C^{'})-\mu_{\boldsymbol\alpha}(C^{''}),
\end{align}
\begin{align}\label{cokernel(d)}
\mu(\coker(d))&\geq\mu_{\boldsymbol\alpha}(C^{'})-\mu_{\boldsymbol\alpha}(C^{''})+2g-2
\end{align}
where $d: \mathcal{H}^0\longrightarrow \mathcal{H}^{1}$ is the
morphism in the $\Hom$-complex $\mathcal{H}^\bullet(C^{''},C')$, defined in
\eqref{deformation complex of chains}.
\end{theorem}

\begin{proof}
  Denote the quiver
  bundle~\eqref{quiver1} by $\bar{\mathcal{E}}$.
  Using
  $\ker(d)$, define a subobject of $\bar{\mathcal{E}}$ as
  follows:
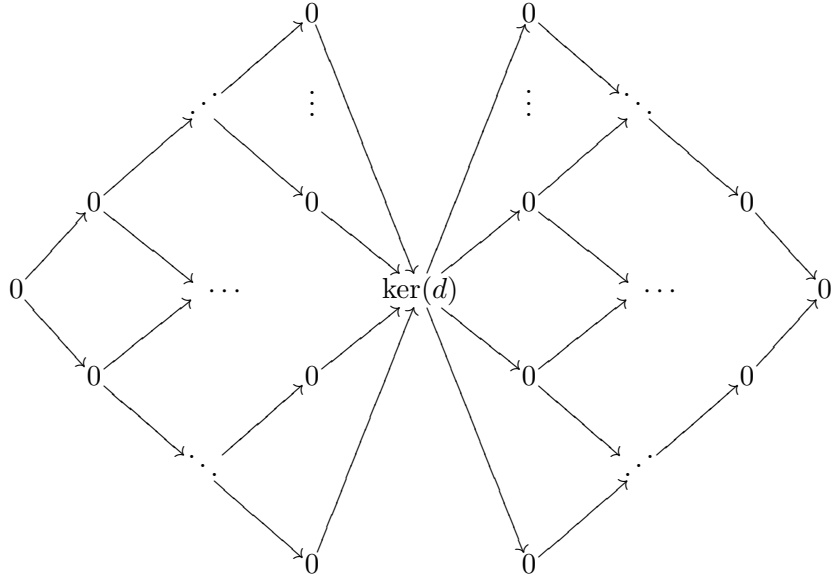
\begin{equation}
 \xymatrix@=2em@1@R=2em{
   & & & &0\ar[dddr]& &0\ar[dr]& & &\\
      & & & \iddots\ar[ur]\ar[dr]& \vdots& & \vdots& \ddots\ar[dr]&& \\
        & & 0\ar[ur]\ar[dr]& &0\ar[dr]& &0\ar[ur] \ar[dr]&&0\ar[dr] \\
          & 0\ar[ur]\ar[dr]& & \mbox{ }\mbox{ }\mbox{ }\mbox{  }\cdots& & \ker(d)\ar[ur]\ar[dr]\ar[uuur]\ar[dddr]& & \mbox{ }\mbox{ }\mbox{ }\mbox{  }\cdots & &0\\
            & &  0\ar[ur]\ar[dr]& & 0\ar[ur]& & 0\ar[ur]\ar[dr]&& 0\ar[ur]& \\
                     & & &\ddots\ar[ur] \ar[dr]&& && \iddots\ar[ur] && \\
                    & & & &0\ar[uuur]&& 0\ar[ur]&& }
                    \end{equation}
By Theorem~\ref{cor:relating-polystability} $\bar{\mathcal{E}}$ is
  $\bar{\boldsymbol\alpha}$-semistable, and a simple calculation shows
  that $\mu_{\bar{\boldsymbol\alpha}}(\bar{\mathcal{E}})
      =\mu_{\boldsymbol\alpha}(C^{'})-\mu_{\boldsymbol\alpha}(C^{''})+\alpha_m$. Hence
 \begin{align*}
 \mu(\ker(d))+\alpha_m\leq
   \mu_{\boldsymbol\alpha}(C^{'})-\mu_{\boldsymbol\alpha}(C^{''})+\alpha_m,
 \end{align*}
which is equivalent to $(\ref{kernel(d)})$. 

To prove $(\ref{cokernel(d)})$ consider the following quotient quiver bundle  
                          
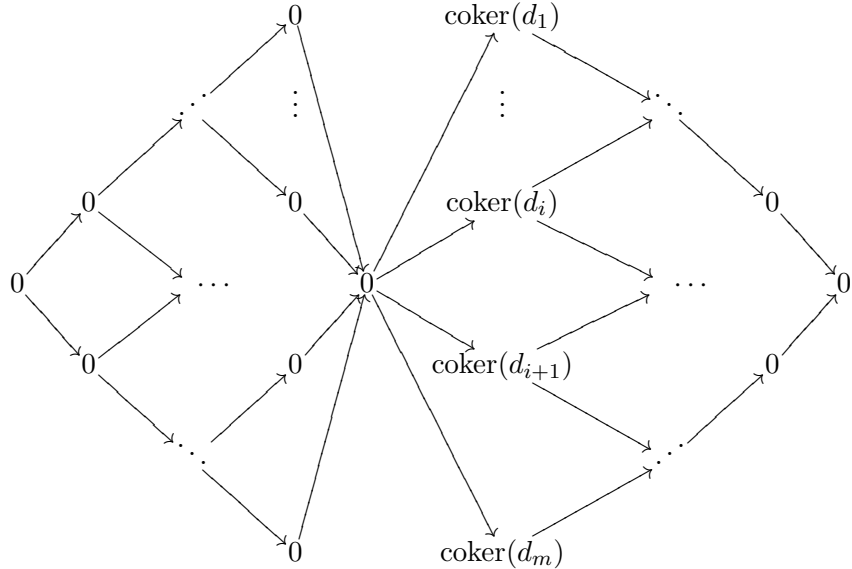
\begin{equation}
\label{eq:quotient}
 \xymatrix@=1.8em@1@R=1.8em{
   & & & &0\ar[dddr]& &\coker(d_1)\ar[dr]& & &\\
      & & & \iddots\ar[ur]\ar[dr]& \vdots& & \vdots& \ddots\ar[dr]&& \\
        & & 0\ar[ur]\ar[dr]& &0\ar[dr]& &\coker(d_{i})\ar[ur] \ar[dr]&&0\ar[dr] \\
          & 0\ar[ur]\ar[dr]& & \mbox{ }\mbox{ }\mbox{ }\mbox{  }\cdots& & 0\ar[ur]\ar[dr]\ar[uuur]\ar[dddr]& & \mbox{ }\mbox{ }\mbox{ }\mbox{  }\cdots & &0\\
            & &  0\ar[ur]\ar[dr]& & 0\ar[ur]& & \coker(d_{i+1})\ar[ur]\ar[dr]&& 0\ar[ur]& \\
                     & & &\ddots\ar[ur] \ar[dr]&& && \iddots\ar[ur] && \\
                    & & & &0\ar[uuur]&& \coker(d_{m})\ar[ur]&& }
                    \end{equation}
                    where $d_i$ is defined as the composition: 
                    $$d_i: {\textstyle\underset{i-j=0}{\bigoplus}\Hom(E_i'',E'_j)
  \xrightarrow{\delta^++\delta^-}
  \underset{i-j=1}{\bigoplus}\Hom(E_i'', E'_j)}\to \Hom(E_i'', E'_{i-1}).$$
By $\bar{\boldsymbol\alpha}$-semistability of 
$\bar{\mathcal{E}}$ the $\bar{\boldsymbol\alpha}$-slope of
\eqref{eq:quotient} is greater than or equal to
$\mu_{\bar{\boldsymbol\alpha}}(\bar{\mathcal{E}})$, which means that
$$\mu(\coker(d))+\alpha_m+\frac{\sum_{i=1}^{m}(\alpha_{i}-\alpha_{i-1})\rk(\coker(d_i))}{\sum_{i=1}^{m}\rk(\coker(d_i)}\geq\mu_{\boldsymbol\alpha}(C^{'})-\mu_{\boldsymbol\alpha}(C^{''})+\alpha_m
$$
and therefore
 \begin{align*} \mu(\coker(d))&\geq\mu_{\boldsymbol\alpha}(C^{'})-\mu_{\boldsymbol\alpha}(C^{''})+\frac{\sum_{i=1}^{m}(\alpha_{i}-\alpha_{i-1})\rk(\coker(d_i))}{\sum_{i=1}^{m}\rk(\coker(d_i)}\\
 &\geq \mu_{\boldsymbol\alpha}(C^{'})-\mu_{\boldsymbol\alpha}(C^{''})+2g-2,
 \end{align*}
which gives \eqref{cokernel(d)}.
  \end{proof}

\end{document}